\documentclass[12pt]{article}
 \usepackage[a4paper]{anysize}\marginsize{2cm}{2cm}{2cm}{2cm}
 \pdfpagewidth=\paperwidth \pdfpageheight=\paperheight
 \usepackage{amsfonts,amssymb,amsthm,amsmath,eucal,tabu,url}
 \usepackage{pgf}
 \usepackage{array}
 \usepackage{pstricks}
 \usepackage{pstricks-add}
\usepackage{pgf,tikz,pgfplots}
\pgfplotsset{compat=1.10}
 \usetikzlibrary{automata}
 \usetikzlibrary{arrows}
 \usepackage{indentfirst}
\theoremstyle{plain}
\newtheorem{thm}{Theorem}[section]
\newtheorem{theorem}[thm]{Theorem}

\newtheorem{proposition}[thm]{Proposition}

\newtheorem{conjecture}[thm]{Conjecture}

\theoremstyle{definition}
\newtheorem{definition}[thm]{Definition}
\newtheorem{remark}[thm]{Remark}
\newtheorem{example}[thm]{Example}

\newtheorem{thevarthm}[thm]{\varthmname}

\newenvironment{varthm*}[1]{\trivlist\item[]{\bf #1.}\it}{\endtrivlist}

\def\keywordname{{\bfseries Keywords}}%
\def\keywords#1{\par\addvspace\medskipamount{\rightskip=0pt plus1cm
\def\and{\ifhmode\unskip\nobreak\fi\ $\cdot$
}\noindent\keywordname\enspace\ignorespaces#1\par}}
\def\subclassname{{\bfseries Mathematics Subject Classification
(2010)}\enspace}
\def\subclass#1{\par\addvspace\medskipamount{\rightskip=0pt plus1cm
\def\and{\ifhmode\unskip\nobreak\fi\ $\cdot$
}\noindent\subclassname\ignorespaces#1\par}}

\def\P{\mathbb{P}}

\title{On $3$-syzygy and unexpected plane curves}
\author{Grzegorz Malara, Piotr Pokora\footnote{Corresponding Author} \,and Halszka Tutaj-Gasi\'nska}
\date{\today}
\begin{document}
\maketitle
\begin{abstract}
In this note we study  curves (arrangements) in the complex projective plane which can be considered as generalizations of \emph{free curves}. We construct families of  arrangements which are nearly free and possess interesting geometric properties. More generally, we study $3$-syzygy curve arrangements and we present examples that admit unexpected curves.

\subclass{14N20, 14C20, 32S22}
\end{abstract}
\section{Introduction}\label{sec:intro}
In the recent years there is an increasing interest in free divisors in the complex projective plane. Let $C$ be a reduced plane curve of degree $d$ given by a homogeneous polynomial $f \in S := \mathbb{C}[x,y,z]$. Let us consider the following short exact sequence
$$0 \rightarrow \mathcal{D}_{0} \rightarrow \mathcal{O}_{\mathbb{P}^{2}_{\mathbb{C}}}^{3} \xrightarrow{(\partial_{x}f, \, \partial_{y}f, \, \partial_{z}f)} \mathcal{J}_{f}(d-1) \rightarrow 0,$$
where the sheaf $\mathcal{J}_{f}$ is the Jacobian ideal of $f$.
The sheaf $\mathcal{D}_{0}$ is reflexive and has rank two, therefore it is a vector bundle on $\mathbb{P}^{2}_{\mathbb{C}}$. 
\begin{definition}
A reduced plane curve $C \subset \mathbb{P}^{2}_{\mathbb{C}}$ is \emph{free}
with the exponents $(d_{1},d_{2}) \in \mathbb{N}^{2}$ and $0 \leq d_{1} \leq d_{2}$ if the associated vector bundle $\mathcal{D}_{0}$ splits as
$$\mathcal{D}_{0} = \mathcal{O}_{\mathbb{P}^{2}_{\mathbb{C}}}(-d_{1}) \oplus \mathcal{O}_{\mathbb{P}^{2}_{\mathbb{C}}}(-d_{2}).$$
\end{definition}
One of the most important questions in the context of free divisors is the famous Terao conjecture \cite[Conjecture 4.138]{OT92}. Assume that $\mathcal{L}$ is an arrangement of $d$ lines in $\mathbb{P}^{2}_{\mathbb{C}}$. We define the Levi graph $G$ for $\mathcal{L}$ as a bipartite graph having one vertex $x_{i}$ per each line $\ell_{i}$ and one vertex $y_{j}$ per each intersection point $p_{j} \in {\rm Sing}(\mathcal{L})$ with an edge $e_{ij}$ if $p_{j}$ is incident with $\ell_{i}$. 
\begin{conjecture}[Terao]
Let $\mathcal{L}_{1}, \mathcal{L}_{2} \subset \mathbb{P}^{2}_{\mathbb{C}}$ be two line arrangements such that their associated Levi graphs $G_{1}$ and $G_{2}$ are isomorphic. Then $\mathcal{L}_{1}$ is free if and only if $\mathcal{L}_{2}$ is free.
\end{conjecture}
Let us point out that Terao's conjecture is widely open. It was shown to be true for $d\leq 13$ lines by Dimca \emph{et al}. in \cite{DIM}. In the context of potential counterexamples to Terao's conjecture, Dimca and Sticlaru defined a new class of arrangements which are called \emph{nearly free}. It is worth mentioning that the concept of nearly free curves appeared before \cite{DIM} in the framework of a naive generalization of Terao's conjecture to conic-line arrangements in the complex projective plane. Schenck and Toh\v aneanu showed in  \cite[Example 4.1]{Schenck1} that such a generalization is false: they constructed two relatively simple arrangements of $2$ lines and $2$ conics having the same combinatorics, where one arrangement is free and another one is nearly free.

For a reduced plane curve $C$ of degree $d$ given by $f \in S$ we denote by $J_{f} = \langle \partial_{x} f, \, \partial_{y} \, f,\partial_{z} \, f \rangle$ the Jacobian ideal and by $\mathfrak{m} = \langle x,y,z \rangle$ the irrelevant ideal. Consider the graded $S$-module $N(f) = I_{f} / J_{f}$, where $I_{f}$ is the saturation of $J_{f}$ with respect to $\mathfrak{m}$.
\begin{definition}
We say that a reduced plane curve $C$ is \emph{nearly free} if $N(f) \neq 0$ and for every $k$ one has ${\rm dim} \, N(f)_{k} \leq 1$. 
\end{definition}
A usual way to define nearly free curves is via the notion of \emph{defect}.
\begin{definition}\label{nu}
Let $C$ be a reduced plane curve. Then we define the defect of $C$ as
$$\nu(C) = {\rm max}_{k} \{ {\rm dim}_{\mathbb{C}} \, N(f)_{k}\}.$$
\end{definition}
We will denote
$${\rm dim}_{\mathbb{C}} \, N(f)_{k}=n(f)_k.$$

Clearly $C$ is a free curve if $\nu(C) = 0$ and $D$ is nearly free if $\nu(D) = 1$.\\
Given a reduced plane curve $C$ defined by $f=0$ in $\mathbb{P}^{2}_{\mathbb{C}}$, then using a result due to Dimca \cite{Dimca1} we can compute explicitly the defect. Let us denote by $\textrm{mdr}(f) = m$ the minimal degree of a non-trivial Jacobian relation, i.e.,
$$a \cdot \partial_{x}\,f + b \cdot \partial_{y} \, f + c \cdot \partial_{z} \, f = 0$$
where coefficients are homogeneous polynomials $a,b,c$ of degree $m$. 

Recall from \cite[p. 526]{G-P} the following crucial definitions.

\begin{definition}

Let $p$ be an  isolated singularity of a polynomial $f\in \mathbb{C}[x,y]. $

We can take local coordinates such that $p=(0,0)$.

The number 
$$\mu_{p}=\dim_\mathbb{C}\left(\mathbb{C}[x,y] /\bigg\langle \frac{\partial f}{\partial x},\frac{\partial f}{\partial y} \bigg\rangle\right)$$
is called the Milnor number of $f$ at $p$.

The number
$$\tau_{p}=\dim_\mathbb{C}\left(\mathbb{C}[ x,y] /\bigg\langle f,\frac{\partial f}{\partial x},\frac{\partial f}{\partial y}\bigg\rangle \right)$$
is called the Tjurina number of $f$ at $p$.
\end{definition}
For a projective situation,
with a point $p\in \mathbb{P}^2$ and a homogeneous polynomial  $f\in \mathbb{C}[x,y,z]$, we take local affine coordinates such that $p=(0,0,1)$ and then the dehomogenization of $f$.

We denote by $\tau(C)$ the total Tjurina number of $C$, i.e., 
$$\tau(C) = \sum_{p \in {\rm Sing}(C)} \tau_{p}.$$

\begin{theorem}[Dimca]\label{thm:defect}
Let $C$ be a reduced plane curve given by $f=0$ of degree $d$ and let $r = {\rm mdr}(f)$. Then we have the following:
\begin{enumerate}
    \item[a)] If $r < \frac{d}{2}$, then $\nu(C) = (d-1)^{2} - r(d-r-1) - \tau(C)$.
    \item[b)] If $r \geq \frac{d-2}{2}$, then
    $$ \nu(C) = \bigg\lceil \frac{3}{4}(d-1)^{2} \bigg\rceil - \tau(C).$$
\end{enumerate}
\end{theorem}
A lot of work has been done to understand geometrical and combinatorial properties of nearly free curves, both from a viewpoint of vector bundles \cite{MarVal} and homological properties of those curves \cite{DimcaSticlaru}.

Recall that  for a curve $C$ given by $f \in S$ we define the Milnor algebra as $M(f) = S / J_{f}$.
The description of $M(f)$ for nearly free curves comes from \cite{DimcaSticlaru} as follows.
\begin{theorem}[Dimca-Sticlaru]
If $C$ is a nearly free curve of degree $d$ given by $f \in S$, then the minimal resolution of the Milnor algebra $M(f)$ has the following form:
$$ 0 \rightarrow S(-b-2(d-1))\rightarrow S(-d_{1}-(d-1))\oplus S(-d_{2}-(d-1)) \oplus S(-d_{3}-(d-1)) \rightarrow S^{3}(-d+1)\rightarrow S$$ for some integers $d_{1},d_{2},d_{3}, b$ such that $d_{1} + d_{2} = d$, $d_{2} = d_{3}$, and $b=d_{2}-d+2$. In that case, the pair $(d_{1},d_{2})$ is called the set of exponents of the nearly free curve $C$.
\end{theorem}

The main purpose of the paper is to construct new examples of plane curves which are  $3$-syzygy curves, see Definition \ref{df:k-syz} -- these are, for instance, nearly free curves. We will use geometry standing behind them in order to show the existence of unexpected curves.
Our results concerning unexpected curves provide a framework for constructing examples of such curves, and give additional insight into the results of Cook II-Harbourne-Migliore-Nagel \cite{CHMN} and of Trok \cite{Trok}.

The paper is organized as follows.
In Section \ref{sec:DelFerArr} using the so-called Fermat arrangements of lines in the complex projective plane we construct an infinite family of nearly free curves (a deletion type family) with the property that this family admits a series of unexpected curves (what is shown in Section \ref{sec:unexp}).  In Section \ref{sec:nearfreecon} we  construct a family of conic arrangements which are nearly free (an addition type arrangement) and in Section \ref{sec:almfree} we present an example of $4$ conics with $12$ tacnodes which is almost free. 
In Section \ref{sec:unexp} we study some $3$-syzygy arrangements and show that they admit unexpected curves.
Then, in Section \ref{sec:maclane}, we focus on conic arrangement constructed via flat extensions. Our procedure leads to an interesting arrangement of smooth conics, but it turns out that flat extensions do not preserve the nearly freeness. However, we believe that the construction is interesting on its own right, and this approach fits into the current interest of researchers in constructing new curve arrangements. The Leitmotif for our work is to explore addition-deletion type constructions of curve arrangements in the plane since this topic is still open in its whole generality. In this context it is worth mentioning that an addition-deletion type statement was proved by Schenck, Terao, and Yoshinaga \cite{Schenck2} for arrangements of plane curves with quasihomogeneous singularities.

\section{Deletion for Fermat arrangements}\label{sec:DelFerArr}
In this section we are going to investigate the addition-deletion procedure starting with free arrangements of lines. Roughly speaking, we want to construct new examples (or even a new family) of nearly free arrangements constructed by deleting some lines from free arrangements of lines. This path of investigations was indicated, for instance, in \cite{MarVal}. Our main result of this section provides a whole family of nearly free arrangements which are non-examples with respect to \cite[Proposition 3.1.]{MarVal}.

We start with the well-known family of Fermat line arrangements $\mathcal{F}_{n}$ in $\mathbb{P}^{2}_{\mathbb{C}}$ given by the following defining polynomial
$$Q(x,y,z) = (x^{n} - y^{n})(y^{n}-z^{n})(z^{n}-x^{n})$$
with $n\geq 3$. This arrangement has exactly $n^{2}$ triple points and $3$ points of multiplicity $n$, and it is known to be free with the exponents $(n+1,2n-2)$.
We consider the arrangement $\mathcal{NF}_{n}$ defined by the following equation
$$\widetilde{Q}(x,y,z) = (x^{n} - y^{n})(y^{n}-z^{n})(z^{n}-x^{n})/(x-y) = (x^{n-1} + yx^{n-2} + ... +y^{n-2}x + y^{n-1})(y^{n}-z^{n})(z^{n}-x^{n})$$
with $n\geq 3$.

Let us present the most important combinatorial properties of $\mathcal{NF}_{n}$.
\begin{enumerate}
    \item For $n=3$ we obtain an arrangement which is isomorphic to the famous MacLane arrangement of $8$ lines -- according to our best knowledge this fact is not written explicitly in the literature. The arrangement of $8$ lines with the maximal number of triple points is considered, for instance, in \cite{triple}.
    \item For given $n\geq 3$ the arrangement $\mathcal{NF}_{n}$ has exactly two points of multiplicity $n$, one point of multiplicity $n-1$, $n^{2}-n$ triple points, and $n$ double points.
\end{enumerate}
Now we show the following main result of this section.
\begin{proposition}
For $n\geq 3$ the arrangement $\mathcal{NF}_{n}$ is nearly free with the exponents $d_{1} = n+1$ and $d_{2}=2n-2$.
\end{proposition}
\begin{proof}
First of all, using \cite[Theorem 5.11]{abe-dimca} and \cite[Theorem 1.4]{abe18}, we write down the minimal free resolution of the Milnor algebra $M(\widetilde{Q})$, which has the following form
$$0 \rightarrow S(-5n+3) \rightarrow S^{2}(-5n+4) \oplus S(-4n+1) \rightarrow S^{3}(-3n+2) \rightarrow S,$$
and we can conclude that $\textrm{mdr}(\widetilde{Q}) = n+1$. 
We can compute the total Tjurina number by the following formula
$$\tau(\mathcal{NF}_{n}) = \sum_{p \in {\rm Sing}(\mathcal{NF}_{n})} \mu_{p} = \sum_{p \in {\rm Sing}(\mathcal{NF}_{n})} ({\rm mult}_{p} - 1)^{2} = 7n^{2} -11 n + 6,$$
where the first equality follows from the fact that all singular points of line arrangements are quasihomogenous and then by  \cite[Satz]{Saito} we have $\mu_p = \tau_p$, the second equality is a well-known property of Milnor numbers when singularities are ordinary, and it comes directly from \cite[Theorem 10.5]{Milnor}.

In order to finish our proof, we need to consider two cases, namely when $n = 3$ and $n\geq 4$. We use Theorem \ref{thm:defect}. If $n\geq 4$, then $\textrm{mdr}(\widetilde{Q}) < \frac{3n}{2}$, and
$$\nu(\mathcal{NF}_{n}) = (3n-2)^{2} - (n+1)\cdot(2n-3) - (7n^{2}-11n+6) = 1.$$
If $n=3$, then $4 = \textrm{mdr}(\widetilde{Q})  > \frac{d-2}{2} = 3$, and thus we can apply the second alternative in Theorem \ref{thm:defect}, namely
$$\nu(\mathcal{NF}_{3}) = \bigg\lceil \frac{3}{4}(8-1)^{2} \bigg\rceil - 36 = 1,$$
which completes the proof.

\end{proof}


\section{Nearly free conic arrangements}\label{sec:nearfreecon}
In this section we would like to construct interesting arrangements of smooth conics which are nearly free. Our main tool is the classical Kummer cover of the projective plane, namely
$$\pi_{k}((x:y:z)) := (x^{k} : y^{k} : z^{k})$$
with $k \geq 2$. Kummer covers are finite Galois covers of $\mathbb{P}^{2}_{\mathbb{C}}$ ramified along $xyz=0$ with the Galois group $(\mathbb{Z}/k\mathbb{Z})^{2}$ and these are very useful when we want to construct an interesting curve starting from a simpler one. Let us consider the curve $C$ given by $h=x^{2}+y^{2}+z^{2} - 2(xy + yz + yz)$. We define a new curve by
$$C_{k} : \pi_{k}^{*}(h)=x^{2k}+y^{2k}+z^{2k} - 2(x^{k}y^{k} + y^{k}z^{k} + y^{k}z^{k})  = 0.$$
As it was shown by Artal Bartolo \emph{et al.} in \cite{Artal} that the curve $C_{k}$ is nearly free. A similar idea to use Kummer covers appeared in \cite{PR}, but in a different setting of curves with low Harbourne indices. Using this idea we are going to construct a family of conic arrangements $\mathcal{C}_{k}$ such that for every $k\geq 2$ the resulting curve is nearly free. In this particular case, the Kummer cover plays a role of the conductor for the addition procedure for curve arrangements.

Let us present our construction. Our starting point is the following smooth conic $C: xy + z^{2}=0$. Then we apply the Kummer cover of order $k\geq 2$ obtaining the following family 
$$\mathcal{C}_{k} \, : \, \{ f_{k} := x^{k}y^{k} + z^{2k} =0\}.$$ 
Observe that $\mathcal{C}_{k}$ splits as an arrangement of $k$ smooth conics. It is easy to observe that $\mathcal{C}_{k}$ has only two singular points which are $P_{1} = (1:0:0)$ and $P_{2} = (0:1:0)$. The main result of this section is the following.
\begin{proposition}
The arrangement $\mathcal{C}_{k}$ with $k\geq 2$ is nearly free with the exponents $(1,2k-1)$.
\end{proposition}
\begin{proof}
We can find, just by hand, non-trivial relations among the partial derivatives of $f_{k}$, namely
\begin{equation} \label{eq1}
\begin{split}
(\triangle) & : x \cdot \partial_{x} \,f_{k} - y \cdot \partial_{y} \, f_{k} = 0, \\
(\diamondsuit)  & : 2 z^{2k-1} \cdot \partial_{x} \, f_{k} - x^{k-1}y^{k} \cdot \partial_{z} \, f_{k} = 0, \\
(\star)     & : 2 z^{2k-1} \cdot \partial_{y} \, f_{k} - x^{k}y^{k-1} \cdot \partial_{z} \, f_{k} = 0.
\end{split}
\end{equation}
Our aim is to apply Theorem \ref{thm:defect}. First of all, $\textrm{mdr}(f_{k}) = 1$. Moreover, observe that for each singular point $P_{i}$ we have locally, after dehomogenization and a possible change of coordinates, the polynomial $x^{k} + y^{2k}$, which allows us to compute the Tjurina number for each $P_{i}$, so we have $\tau_{P_{i}} = (2k-1)(k-1)$. We obtain
$$\nu(\mathcal{C}_{k}) = (2k-1)^{2} - (2k-2) - 2\cdot(2k-1)(k-1) = 1,$$
so arrangement $\mathcal{C}_{k}$ is nearly free. We can find by hand the minimal free resolution of the Milnor algebra $M(f_{k})$ which has the following form:
$$0 \rightarrow S(-4k+1) \rightarrow S(-4k+2)^{2} \oplus S(-2k) \rightarrow  S^{3}(-2k+1) \rightarrow S,$$
so the exponents of the nearly free arrangement $\mathcal{C}_{k}$ are $(1, 2k-1)$.
\end{proof}
\begin{remark}
The idea to use conics of a similar type appears in \cite{Dimca2} where the author constructs also examples of nearly free curves.
\end{remark}
At the end of this section, let us present some experiments that we performed using Singular, \cite{DGPS}. If we add the line $x=0$ or $y=0$ to the arrangement $\mathcal{C}_{k}$ we obtain a free arrangement with the exponents $(1,2k-1)$. Adding the line $z=0$ in return gives a nearly free arrangement with the same exponents. We see that the result is connected with the number of singularities of $\mathcal{C}_{k}$ that belong to the added line. These insights suggest that a similar statements, as in \cite{Schenck1} for conic-line free arrangements, can be investigated in the nearly free cases.

\section{Almost free curves}\label{sec:almfree}
In this section we would like to add to the list some new examples of \emph{almost free} curves. 

\begin{definition}[\cite{DimcaSticlaru}]
 We say that a reduced plane curve $C$ of degree $d$ given by $f=0$ is \emph{almost free} if ${\rm dim} \, N(f) = 1$ for $d$ even, and ${\rm dim} \, N(f) = 2$ for $d$ odd.
\end{definition}
\begin{example}
As a warm-up, let us consider the following line arrangement $\mathcal{A}_{9}$ in $\mathbb{P}^{2}_{\mathbb{C}}$ given by 
$$Q(x,y,z) = x(x^{2}+xy+y^{2})(y^{3} - z^{3})(z^{3} -x^{3}).$$
This arrangement can be viewed as an addition arrangement since we add one line $x=0$ to the MacLane arrangement. 
By an easy inspection (supported by Singular computations, \cite{DGPS}) we can see that the minimal resolution of the Milnor algebra $M(Q)$ has the following form
$$0 \rightarrow S(-14) \rightarrow S^{2}(-13) \oplus S(-12) \rightarrow S^{3}(-8) \rightarrow S.$$
Using the results from \cite{HS12} (see Subsection \ref{subsec:plus} for more detailed description of that results) we get the resolution of $N(Q)$:
$$0 \rightarrow S(-14) \rightarrow S^{2}(-13) \oplus S(-12) \rightarrow    S^{2}(-11)\oplus S(-12)\rightarrow S(-10).$$
Now by the formula for the Hilbert function (cf. \cite{eisenbud}, Corollary 1.2) we have 
${\rm dim}N(Q)_{10} = {\rm dim} N(Q)_{11} = 1$, and ${\rm dim} N(Q)_{k} = 0$ for $k\not\in \{10,11\}$, so the arrangement $\mathcal{A}_{9}$ is almost free. It is worth noticing that if we take a slightly more general line in the place of 
$x=0$, for instance $x+y+z=0$, then for this choice the resulting arrangement $\mathcal{A}_{9}'$ is free with the exponents $d_{1} = 4$ and $d_{2}=4$.
\end{example}
Now we would like to present an interesting example of an irreducible almost free curve of degree $10$. This curve does not fit into the world of arrangements, but we believe that it is worth presenting here due to its possible further utility.
\begin{example}
The following curve is strictly related with Hirzebruch's works on Hilbert modular surfaces. Let us consider the Klein curve $\mathcal{K}$ of degree $10$ given by the following equation
\begin{multline*}
K(x,y,z) = 320x^{2}y^{2}z^{6} - 160x^{3}y^{3}z^{4} + 20x^{4}y^{4}z^{2} + 6x^{5}y^{5} + x^{10} + y^{10} \\
 - 4(x^{5}+y^{5})\cdot(32z^{5}-20xyz^{3}+5x^{2}y^{2}z).    
\end{multline*}
The curve $\mathcal{K}$ has exactly six singular points which are double $(2,3)$-cusps -- these are locally given by the following equation $(x^{3} +y^{2})(y^{3}+x^{2}) = 0$. Moreover, $\mathcal{K}$ is invariant under the action of ${\rm PSL}(2,5) = A_{5}$, and the configuration of these six cusps forms  a  unique  minimal  orbit  of $A_{5}$ -- see \cite{Klein}. The minimal free resolution of $M(K)$ has the following form
$$0 \rightarrow S(-15) \rightarrow S(-14)^{3} \rightarrow S(-9)^{3} \rightarrow S,$$
${\rm dim}N(K)_{12} =1$, and ${\rm dim}N(K)_{k} = 0$ for $k\neq 12$, so our curve $\mathcal{K}$ is almost free.
\end{example}
Finally we would like to look at the question whether arrangements of curves with points of multiplicity $2$ can be either nearly or almost free. In the case of $d$ generic lines $\mathcal{L}_{d}$ with $\binom{d}{2}$ double intersection points, the only interesting case is when $d=3$ and then the arrangement $\mathcal{L}_{3}$ is free. Here we focus on arrangements of conics with only tacnodes as the intersection points -- these are non-ordinary points of multiplicity $2$.
\begin{example}
There is an interesting problem of the existence of conic arrangements possessing the maximal possible number of tacnodes. There is an obvious upper bound 
$$\#\,  {\rm tacnodes} \leq k(k-1),$$
where $k$ denotes the number of conics.
Let us consider the arrangement $\mathcal{C} = \{C_{1},C_{2},C_{3},C_{4}\} \subset\mathbb{P}^{2}_{\mathbb{C}}$ given by $C=f_{1}\cdot ... \cdot f_{4} = 0$ such that
$$C_{1} : f_{1} = xy-z^{2},$$
$$C_{2} : f_{2} = xy+z^{2},$$
$$C_{3} : f_{3} = x^{2} + y^{2} -2z^{2},$$
$$C_{4} : f_{4} = x^{2} + y^{2} +2z^{2}.$$
One can easily check that $\mathcal{C}$ has exactly $12$ tacnodes, so for $k=4$ we obtained the maximal possible number of tacnodes according to the above naive upper-bound.
We can compute the minimal free resolution of the Milnor algebra $M(C)$ which has the following form
$$0 \rightarrow S(-12) \rightarrow S^{3}(-11) \rightarrow S^{3}(-7) \rightarrow S,$$
${\rm dim}N(C)_{9} =1$, and ${\rm dim}N(C)_{\ell} = 0$ for $\ell \neq 9$, so our arrangement $\mathcal{C}$ is almost free. It is worth noticing that if we consider an arrangement of $5$ conics with $17$ tacnodes (which is the maximal possible number of tacnodes in that case by the Miyaoka-Yau inequality, see for instance \cite{megyesi}), then this arrangement is neither nearly free nor almost free.
\end{example}

\section{Unexpected curves}\label{sec:unexp}

Let $Z = P_{1}+ ... + P_{s}$ be a reduced scheme of mutually distinct points in $\mathbb{P}^{2}_{\mathbb{C}}$. We say that $Z$ admits an unexpected curve of degree $d$ if for a generic point $P \in \mathbb{P}^{2}_{\mathbb{C}}$ of multiplicity $m$ we have that
\begin{equation}\label{ineq-unexpected}
{\rm dim}_{\mathbb{C}} [I(Z + mP)]_{d} > {\rm max} \bigg\{ {\rm dim}_{\mathbb{C}} [I(Z)]_{d} - \binom{m+1}{2}, 0 \bigg\},
\end{equation}
where $I(Z + mP) = I(P_{1})\cap ... \cap I(P_{s}) \cap I(P)^{m}$.

A general question that we can ask is to classify all those configurations of mutually distinct points $Z$ that admit unexpected curves. In the pioneering paper \cite{CHMN}, Cook II, Harbourne, Migliore, and Nagel study unexpected curves from the viewpoint of line arrangements in the complex projective plane, i.e., in their setting $Z$ denotes the set of points which are dual to lines of a given arrangement $\mathcal{A} \subset \mathbb{P}^{2}_{\mathbb{C}}$. It is worth emphasizing here that there are different approaches towards the unexpected curves. In the first approach, the authors focus on unexpected curves of degree $d$ having at a general point $P$ multiplicity $d-1$, but  we can also consider a slightly different scenario (according to our definition presented here). In order to abbreviate, we say that a set $Z$ has the $U(2,d,m)$-property if $Z\in \mathbb{P}^2$ admits an unexpected curve of degree $d$ having at a general point $P$ multiplicity $m$.

In \cite{CHMN}, Cook II, Harbourne, Migliore, and Nagel considered a set of points $Z=\{ z_1,\dots,z_d \}$ in $\mathbb{P}^{2}_{\mathbb{C}}$ and the dual line arrangement $\mathcal{A}_Z=\{\ell_1,\ldots,\ell_d\}$ given by the defining polynomial $f$. The Jacobian ideal of $f$ has its syzygy bundle:
$$0\to \textrm{Syz} \, J_f\to S(d-1)^3\to J_f\to 0.$$
Let us recall that the module of logarithmic derivations $D$ is a submodule of ${\rm Der}(S)$ (the module of all $\mathbb{C}$-linear derivations) consisting of all those elements $\delta \in  {\rm Der}(S)$ such that $\delta(f) \in S \cdot \langle f \rangle$.
Obviously $D$ contains the Euler derivation $\partial_E=x\partial_x+ y\partial_y+z\partial_z$. We know that the quotient $D_0=D/\partial_E$ is isomorphic to  the twist of $ \textrm{Syz} \,J_f$, namely we have the following exact sequence of sheaves
$$0\to {\cal D}_0\to {\cal O}_{\P^{2}_{\mathbb{C}}}^3\to {\cal J}_f(d-1)\to 0,$$
where ${\cal J}_f$ is the sheafification of $J_{f}$ and ${\cal D}_0$ is a locally free sheaf of rank $2$, see the Appendix to \cite{CHMN}.  

It is well-known that $ {\cal D}_0$  restricted to a generic line splits, according to Grothendieck's theorem \cite{Gr}, as a sum of line bundles $ {\cal O}_{\P^{1}_{\mathbb{C}}}(-a_Z)\bigoplus {\cal O}_{\P^{1}_{\mathbb{C}}}(-b_Z)$.  If the line is generic, the pair $(a_Z,b_Z)$ is called the splitting type of  $ {\cal D}_0$ and it satisfies $a_Z+b_Z=|Z|-1$.

Cook II, Harbourne, Migliore and Nagel considered the case when an unexpected curve is of degree $d=m+1$, where $m$ is, as we said above,  the multiplicity that the curve has in a given generic point $P$ (we will say that such unexpected curves have the $U(2,m+1,m)$-property and in what follows an \emph{unexpected curve} means here unexpected curve of type with the  $U(2,m+1,m)$-property).
They proved the following theorem, which we quote in a   version changed according to Dimca's paper \cite{dimca-unexp}.

\begin{theorem}\label{unexp}
Let $Z$ be a finite set of $d$ points in $\mathbb{P}^{2}_{\mathbb{C}}$.  Let $f$ be the polynomial given by the product of lines dual to the points of $Z$ and denote by $\mathcal{A}_{Z}$ the line arrangement given by $f=0$. Let $(a_Z,b_Z)$ be the splitting type of the derivation bundle $ {\cal D}_0$ for $\mathcal{A}_{Z}$.
Let $m(\cal{A}_Z)$ be the  maximal multiplicity among the multiplicities of the points of the arrangement ${\cal A}_Z$. Then $Z$ admits an unexpected curve with the $U(2,m+1,m)$-property if and only if 
$$m({\cal A}_Z)\leq a_Z+1< \frac{|Z|}{2}.$$
If $Z$ admits an unexpected curve, then it admits such a curve of degree $j$ if and only if $j$ satisfies
$$a_Z<j\leq |Z|-a_Z-2.$$
The curve $C$ of minimal possible degree, $a_Z+1$, is unique and it is irreducible iff $a_Z=a_{Z_i}$, where $(a_{Z_i}, b_{Z_i})$ is the splitting type if $D_{0i}$, the bundle of logarithmic derivations for $Z\setminus \{z_i \}, i=1,\dots,d$.
\end{theorem}
 Dimca in \cite{dimca-unexp} proved that in the setting as above $a_Z$ may be replaced by the minimal degree of syzygies $\textrm{mdr}(f)$ and $a_{Z_i}$ by the minimal degree of syzygies from $ \textrm{Syz} \,J_{f_i}$, where $f_i$ is the product of lines dual to points from $Z\setminus\{z_i \}$.

\subsection{Nearly free arrangements}

Let us recall \cite[Example 6.1]{CHMN} which motivates our investigations in this section.

\begin{example}
\label{exp:CHMN}
Consider the arrangement $\mathcal{A}\subset\mathbb{P}^{2}_{\mathbb{C}}$ of $19$ lines given by the following defining polynomial: \\
\begin{align*}
Q(x,y,z) &= xyz(x+y)(x-y)(2x+y)(2x-y)(x+z)(x-z)(y+z)(y-z)(x+2z) \\ &(x-2z)(y+2z)(y-2z)(x-y+z)(x-y-z)(x-y+2z)(x-y-2z).
\end{align*}
We can compute the minimal free resolution of the Milnor algebra $M(Q)$ obtaining
$$0 \rightarrow S(-30) \rightarrow S(-29)^{2} \oplus S(-26) \rightarrow S(-18)^{3} \rightarrow S.$$
This nearly free arrangement, as it was said in \cite{CHMN}, is \emph{close} to be free in the sense of the addition-deletion procedure, namely if we remove the line $2x+y$, then the new arrangement $\mathcal{A}'$ is free with exponents $d_{1}=7$ and $d_{2}=10$. It turns out that the set of duals to lines in $\mathcal{A}$ has $U(2,9,8)$-property, i.e., it admits an unexpected curve of degree $9$ having at a general point $P$ multiplicity $8$. 
\end{example}


Suppose now we have a set $Z=\{z_1,\dots,z_d\}$, such, that the dual lines give a nearly free arrangement ${\cal A}$ with the exponents $(d_1,d_2)$, where $d_1+d_2=d$.

Abe and Dimca in \cite{abe-dimca}, Corollary 3.5, give the characterization of the splitting type of a nearly free line bundle (we quote it with a slightly changed form).

\begin{theorem}[\cite{abe-dimca}, Corollary 3.5]\label{a-d}
An arrangement ${\cal A}$ is nearly free with the exponents $(d_1,d_2)$  if and only if either

\begin{itemize}
    \item[a)]  $d_1=d_2$ and for every line the bundle $ {\cal D}_0$ splits as  $ {\cal O}_{\mathbb{P}^{1}_{\mathbb{C}}}(-(d_1-1))\bigoplus {\cal O}_{\P^{1}_{\mathbb{C}}}(-d_1)$, or
\item[b)] $d_1<d_2$ and for a generic line $ {\cal D}_0$ splits as  $ {\cal O}_{\P^{1}_{\mathbb{C}}}(- d_1 )\bigoplus {\cal O}_{\P^{1}_{\mathbb{C}}}(-(d_2-1))$. 
\end{itemize}
\end{theorem}

\begin{proposition}
Let $Z=\{z_1,\dots,z_d\}\subset \P^{2}_{\mathbb{C}}$ be a set of points such that  the dual lines give a nearly free arrangement ${\cal A}$ with the exponents $(d_1,d_2)$. Then $Z$ admits an unexpected curve  with $U(2,d_1+1,d_1)$-property if and only if $d_2-d_1\geq 3$.
\end{proposition}
 \begin{proof}
 From Theorem \ref{unexp} and from  Theorem \ref{a-d} we have that if 
$$d_2-d_1\geq 3$$ then $Z$ admits an unexpected curve.

Assume that $d_2-d_1\leq 2$, then we have three possibilities. Either exponents of ${\cal A}$ are $(d_1,d_1+2)$, or $(d_1,d_1+1)$, or both of them are equal $d_1$.
Thus from Theorem \ref{a-d} the splitting types $(a_{Z},b_{Z})$ are $(d_1,d_1+1)$, $(d_1,d_1)$ and $(d_1-1,d_1)$ respectively. According to Theorem \ref{unexp} in none of these cases does $Z$ admit unexpected curve.
\end{proof}
\begin{example}
Let us come back to family $\mathcal{NF}_{n}$ with $n\geq 3$. Then the dual set of points $Z_{n}$ to $\mathcal{NF}_{n}$ admits an unexpected curve with $U(2,d_{1}+1,d_{1})$-property if and only if
$$d_{2}-d_{1} = 2n-2 - (n+1) = n-3 \geq 3,$$
so exactly when $n\geq 6$.
\end{example}
\begin{remark}
In the case of Example \ref{exp:CHMN}, the exponents are $d_{1} = 8$ and $d_{2} = 11$, so we have exactly $d_{2}-d_{1}= 3$.
\end{remark}

\subsection{Plus-one generated and $3$-syzygy line arrangements}\label{subsec:plus}

Let us recall the definition of $k$-syzygy curves from \cite{cdi}.  The authors consider there a reduced complex plane curve  $C\subset \P^{2}_{\mathbb{C}}$ given by an equation $f=0$. In our case 
$C$ is given by the lines of the arrangement ${\cal A}_Z$, lines dual to the points $Z\subset \P^{2}_{\mathbb{C}}$. 

\begin{definition}\label{df:k-syz}
We say that  $C$ is a $k$-syzygy curve if the $S$-module $\textrm{Syz}(J_f)$ is minimally generated by $k$ homogeneous syzygies $\{r_1,r_2,\ldots,r_k\}$ of degree $d_i=\deg r_i$, ordered such that $$1\leq d_1\leq d_2\leq\ldots\leq d_k.$$ 

The multiset $(d_1,d_2,\ldots,d_k)$ is called the exponents of the plane curve $C$ and $\{r_1,r_2,\ldots,r_k\}$ is said to be a minimal set of generators for the $S$-module $\textrm{Syz}(J_f).$ 
\end{definition}
In particular,

\begin{itemize}
    \item[a)] a $2$-syzygy curve $C$ is {\emph{free}}, since then the $S$-module $\textrm{Syz}(J_f)$ is a free module of rank $2$;
    \item[b)] a $3$-syzygy curve is said to be {\emph{nearly free}} exactly when $d_3=d_2$ and $d_1+d_2=d$;
    \item[c)] a $3$-syzygy line arrangement is said to be a {\emph{plus-one generated line arrangement}} of level $d_3$ exactly when $d_1+d_2=d$ and $d_3 \geq d_2$.
\end{itemize}

Let ${\cal A}$ be an arrangement of lines in $\P^2_{\mathbb{C}}$. The general form of the minimal resolution for the Milnor algebra $M(f)$ has the following form (see for instance \cite{cdi} and \cite{HS12}):

\begin{equation}\label{RMM}
0\to \displaystyle{\oplus_{i=1}^{m-2}S(-e_i)}\to \displaystyle{\oplus_{i=1}^{m}S(1-d-d_i)}\to S^3(1-d)\to S,
\end{equation} 
with $e_1\leq e_2\leq\ldots\leq e_{m-2}$ and $1 \leq d_1\leq d_2\leq\cdots\leq d_m$. \\
Moreover $$e_j=d+d_{j+2}-1+\epsilon_j,$$ for $j=1,\ldots, m-2$ and some integers $\epsilon_j\geq 1.$

The minimal resolution of $N(f)$ has the following form (see again \cite{cdi} and \cite{HS12}):
 $$0\to \displaystyle{\oplus_{i=1}^{m-2}S(-e_i)}\to\displaystyle{\oplus_{i=1}^mS(-\ell_i)}\to\displaystyle{\oplus_{i=1}^m}S(d_i-2(d-1))\to \displaystyle{\oplus_{i=1}^{m-2}S(e_i-3(d-1))},$$
 where $\ell_i=d+d_i-1$. \\
 Thus the initial degree of $N(f)$, denoted as $\sigma(C)$,
is equal to
\begin{equation}\label{sigma} 
 \sigma(C)=3(d-1)-e_{m-2}=2(d-1)-d_m-\epsilon_{m-2}.
\end{equation} 

We will use the following result \cite[Corollary 1.4]{cdi}.

\begin{theorem}\label{COR10}
 Let $C$ be a plus-one generated curve of degree $d\geq 3$ with $(d_1,d_2,d_3)$, which is not nearly free, i.e., $d_2<d_3.$ Set $k_j=2d-d_j-3$ for $j=1,2,3$ and $T=3d-6$.
 Then one has the following minimal free resolution of $N(f)$ as a graded $S-$module:
 \begin{align*}0\to S(-d-d_3)\to S(-d-d_3+1)\oplus S(-k_1-2)\oplus S(-k_2-2)\to\\ \to S(-k_1-1)\oplus S(-k_2-1)\oplus S(-k_3-1)\to S(-k_3).\end{align*} In particular, $\sigma(C)=k_3<k_2 \leq \frac{T}{2}$ and the numbers $n(f)_j$ are given by following formulas: 
 \begin{enumerate}
     \item[a)] $n(f)_j=0$ for $j<k_3$;
     \item[b)] $n(f)_j=j-k_3+1$ for $k_3\leq j\leq k_2$;
     \item[c)] $n(f)_j=d_3-d_2+1=\nu(C)$ for $k_2\leq j\leq \frac{T}{2}.$
 \end{enumerate}
\end{theorem}

The next result is due to Abe and Dimca \cite[Corollary 3.3]{abe-dimca}.
\begin{theorem}
\label{th:splittingTypeConnection}
Let $C$ be a reduced curve in $\mathbb{P}^{2}_{\mathbb{C}}$ of degree $d$. Recall that
$$\nu(C) = {\rm max}_{k} \{ {\rm dim}_{\mathbb{C}} \, N(f)_{k}\}.$$
Then
 $$\nu(C)=(d-1)^2-\tau(C)-a_Z b_Z.$$
\end{theorem}

The following examples, together with Theorem \ref{unexp},
show the existence of unexpected curves of type $U(2,j,j-1)$
for \emph{plus-one} and for $3$-syzygy line arrangements. 

\begin{example}
\label{ex:18pts}
    Consider the set $Z \subset \mathbb{P}^{2}_{\mathbb{C}}$ of $18$ points
   $$Z=\{(0,1,0), (-1,1,0), (-2,1,0), (-3,1,0), (-3,2,0), (4,0,-1), (1,1,-1),$$
	$$(2,1,-1), (3,1,-1), (4,1,-1), (0,2,-1), (1,2,-1), (2,2,-1), $$
	$$(0,3,-1), (1,3,-1), (-2,3,-1), (-1,3,-1), (-2,4,-1)\},
   $$
   and the dual line arrangement ${\cal A}_Z$ with its defining polynomial $Q$.
   The minimal free resolution of the Milnor algebra $M(Q)$ is
$$0 \rightarrow S(-30) \rightarrow S(-29) \oplus S(-28) \oplus S(-24) \rightarrow S(-17)^{3} \rightarrow S,$$
from which we see the exponents $(d_1,d_2,d_3)=(7,11,12)$. Since one haa $d_1+d_2=18=d$, $\mathcal{A}_{Z}$ is a \emph{plus-one} generated arrangement of level $12$. This line arrangement has $22$ double, $13$ triple, $7$ quadruple, and $5$ quintuple points, so the total Tjurina number of curve $C$ given by $Q=0$ is
$$\tau(C)=22+13\cdot4+7\cdot9+5\cdot16=217.$$
By $c)$ of Theorem \ref{COR10}, $\nu(C)=d_3-d_2+1=2$ and therefore by Theorem \ref{th:splittingTypeConnection} we get $a_Z b_Z=70$. Together with $a_Z + b_Z = 17$, we finally obtain
$a_Z=7$ and $b_Z=10$. Now Theorem \ref{unexp} justifies the existence of unexpected curves of type $U(2,j,j-1)$, for $j \in\{8,9\}$.
\end{example}
\begin{example}
\label{ex:17pts}
Let $Z'$ be a set of $17$ points defined as $Z\setminus\{(-2,1,0)\}$, where $Z$ is the set from Example \ref{ex:18pts}. Similarly, as in the previous example, we can calculate the minimal free resolution of the Milnor algebra which has the following form:
$$0 \rightarrow S(-29) \rightarrow S(-27) \oplus S(-27) \oplus S(-23) \rightarrow S(-16)^{3} \rightarrow S.$$
    We see that $d_1+d_2=18>d=17$, so it is a $3$-syzygy line arrangement, which is not  \emph{plus-one} generated. The arrangement has $20$ double, $14$ triple, $9$ quadruple, and $2$ quintuple intersection points, so the total Tjurina number of $C$ is equal to
$$\tau(C)=20+14\cdot4+9\cdot9+2\cdot16=189.$$
In this example we combine Theorem \ref{th:splittingTypeConnection} and Theorem \ref{thm:defect} in order to get
$$r(d-r-1)=a_Z b_Z=7\cdot 9=63.$$ 
This leads to the following splitting type $(a_Z,b_Z)=(7,9)$. In consequence, we obtain an unexpected curve of type $U(2,8,7)$.

\begin{figure}
    \centering
    \begin{tikzpicture}[line cap=round,line join=round,>=triangle 45,x=7.0cm,y=7.0cm]
    \clip(-0.6477254956195181,-0.10829882019944272) rectangle (1.1456908411203948,1.1087580062279405);
    \draw [line width=1.pt,domain=-0.6477254956195181:1.1456908411203948] plot(\x,{(-0.-1.*\x)/-0.3333333333333333});
    \draw [line width=1.pt,dashed,domain=-0.6477254956195181:1.1456908411203948] plot(\x,{(-0.-1.*\x)/-0.5});
    \draw [line width=1.pt,domain=-0.6477254956195181:1.1456908411203948] plot(\x,{(-0.-1.*\x)/-1.});
    \draw [line width=1.pt,domain=-0.6477254956195181:1.1456908411203948] plot(\x,{(-0.-1.*\x)/-0.6666666666666666});
    \draw [line width=1.pt,domain=-0.6477254956195181:1.1456908411203948] plot(\x,{(-0.-0.*\x)/1.});
    \draw [line width=1.pt,domain=-0.6477254956195181:1.1456908411203948] plot(\x,{(-0.5-1.*\x)/-2.});
    \draw [line width=1.pt,domain=-0.6477254956195181:1.1456908411203948] plot(\x,{(-0.5-1.*\x)/-1.5});
    \draw [line width=1.pt,domain=-0.6477254956195181:1.1456908411203948] plot(\x,{(-1.-1.*\x)/-3.});
    \draw [line width=1.pt,domain=-0.6477254956195181:1.1456908411203948] plot(\x,{(--1.-1.*\x)/3.});
    \draw [line width=1.pt,domain=-0.6477254956195181:1.1456908411203948] plot(\x,{(--1.-1.*\x)/2.});
    \draw [line width=1.pt,domain=-0.6477254956195181:1.1456908411203948] plot(\x,{(--1.-1.*\x)/1.});
    \draw [line width=1.pt,domain=-0.6477254956195181:1.1456908411203948] plot(\x,{(--0.5-1.*\x)/0.5});
    \draw [line width=1.pt,domain=-0.6477254956195181:1.1456908411203948] plot(\x,{(--0.5-1.*\x)/1.});
    \draw [line width=1.pt,domain=-0.6477254956195181:1.1456908411203948] plot(\x,{(--0.3333333333333333-1.*\x)/0.3333333333333333});
    \draw [line width=1.pt,domain=-0.6477254956195181:1.1456908411203948] plot(\x,{(--0.25-1.*\x)/0.25});
    \draw [line width=1.pt,domain=-0.6477254956195181:1.1456908411203948] plot(\x,{(--0.5-0.*\x)/1.});
    \draw [line width=1.pt,domain=-0.6477254956195181:1.1456908411203948] plot(\x,{(--0.3333333333333333-0.*\x)/1.});
    \draw [line width=1.pt] (0.25,-0.10829882019944272) -- (0.25,1.1087580062279405);
    \end{tikzpicture}
    \caption{Affine part of configurations of lines from Example \ref{ex:18pts}, and Example \ref{ex:17pts}. Line define by equation $-2x+y=0$ is indicated as the dashed line.}
    \label{fig:19lines}
\end{figure}
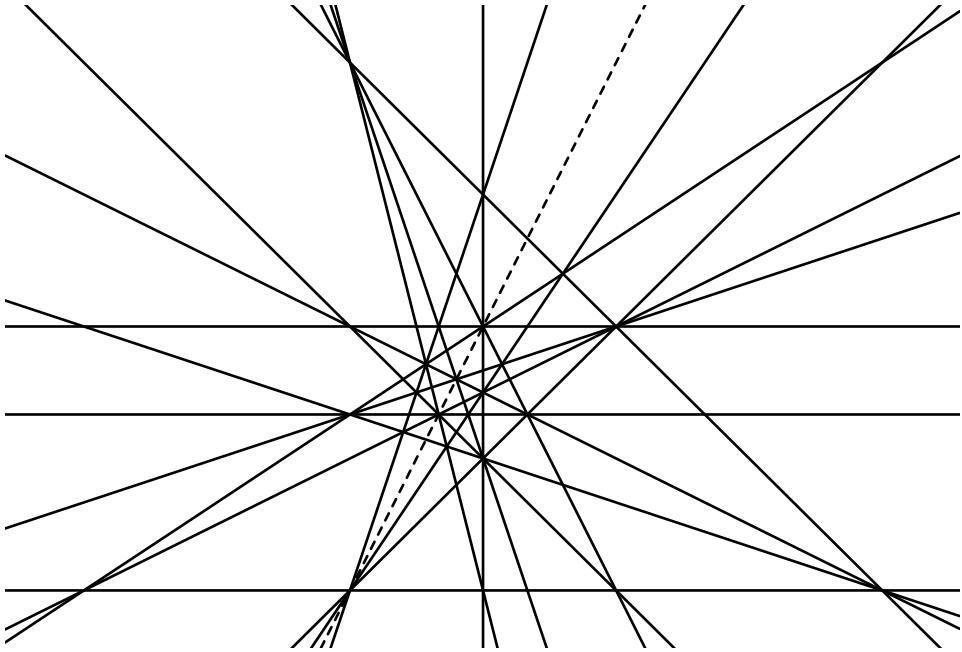
\end{example}
To conclude this section we want to present a whole family of \emph{plus-one} generated arrangements, for which we have unexpected curves.
\begin{example}
    Consider arrangement $\mathcal{NF}_{n}$, defined at the beginning of Section \ref{sec:DelFerArr}. If we remove one line from $\mathcal{NF}_{n}$ we obtain a \emph{plus-one} generated arrangement defined by polynomial $F$, with the following minimal free resolution of $M(F)$:
   $$0 \rightarrow S(-5n+4) \rightarrow S(-5n+5) \oplus S(-5n+6) \oplus S(-4n+2) \rightarrow S^{3}(-3n+3) \rightarrow S.$$
   
   It is not necessary to write down the full minimal free resolution in order to check whether the dual set of points $Z$ has the $U(2,d,m)$-property. If we apply \cite[Proposition 2.12]{ADS19} to $\mathcal{NF}_{n}$, we get immediately that $\textrm{mdr}(F)=n+1$. Then, using Theorem \ref{unexp}, we see that wherever $n \geq 7$ our arrangement given by $F$ admits an unexpected curve of type  $U(2,j,j-1)$, for $j \in \{n+2,\ldots, 2n-5 \}$.
\end{example}

\section{MacLane arrangement of conics via flat extensions}\label{sec:maclane}

Let us consider the MacLane arrangement of $8$ lines in $\mathbb{P}^{2}_{\mathbb{C}}$ given by the following defining equation:
$$Q(x,y,z) = (x^{2} + xy + y^{2})(y^{3}-z^{3})(z^{3}-x^{3}).$$
Our aim is to use a certain flat extension that allows us to construct an arrangement of conics based on the MacLane arrangement.
Consider the following map
$$\Psi: \mathbb{P}^{2}_{\mathbb{C}} \ni (x:y:z) \mapsto (x^{2} + 2y^{2} : y^{2} + 3z^{2} : z^{2} + 5x^{2}) \in \mathbb{P}^{2}_{\mathbb{C}}.$$
It is easy to observe that $\Psi$ is well-defined $4:1$ morphism. Even more is true, namely the set $$E = \{x^{2} + 2y^{2}, y^{2} + 3z^{2}, z^{2} + 5x^{2}\}$$ forms a regular sequence which defines a flat extension.
Our object of studies is an arrangement which is defined by the following equation
$$\Psi^{*}(Q) = 0.$$
As we can see, $\Psi^{*}(Q)$ decomposes into $8$ conics:
\begin{equation*}
\begin{array}{l}
\Psi_{1}(x,y,z) = x^{2}-\frac{1}{2}y^{2}+\frac{1}{4}z^{2}, \\
\Psi_{2}(x,y,z) = x^{2} -\frac{1}{5}y^{2} - \frac{2}{5}z^{2}, \\
\Psi_{3}(x,y,z) =x^{2}+ (e+3)\cdot y^{2}+(3e+3)\cdot z^{2}, \\
\Psi_{4}(x,y,z) = x^{2} +(-e+2)\cdot y^{2} -3e z^{2}, \\
\Psi_{5}(x,y,z) = x^{2}+\bigg(\frac{1}{5}e+ \frac{1}{4}\bigg)\cdot y^{2} + \bigg(\frac{3}{5}e+ \frac{4}{5}\bigg)\cdot z^{2}, \\
\Psi_{6}(x,y,z) = x^{2}- \frac{1}{5} e y^{2} +\bigg(-\frac{3}{5}e+\frac{1}{5}\bigg)\cdot z^{2}, \\
\Psi_{7}(x,y,z) = x^{2}+\bigg(-\frac{10}{31}e+\frac{2}{31}\bigg)\cdot y^{2}+\bigg( \frac{1}{31}e+ \frac{6}{31}\bigg)\cdot z^{2}, \\
\Psi_{8}(x,y,z) =x^{2}+\bigg(\frac{10}{31}e+\frac{12}{31}\bigg)\cdot y^{2}+\bigg(-\frac{1}{31}e+\frac{5}{31}\bigg)\cdot z^{2}, \\
\end{array}
\end{equation*}
where $e$ is a  primitive third root of $1$. Each $\Psi_{i}$ defines a smooth conic in $\mathbb{P}^{2}_{\mathbb{C}}$ and we can show that in fact $\Psi^{*}(Q)$ defines an arrangement of $8$ conics having  $t_{2} = 4\cdot 4 = 16$ and $t_{3} = 4 \cdot 8 = 32$. It is worth noticing that all singular points are quasihomogeneous (and ordinary). We call the arrangement defined by $\Psi^{*}(Q)$ \textit{MacLane arrangement of 8 conics}.
The minimal resolution of the Milnor algebra $M(\Psi^{*}(Q))$ has the following form:
$$0\rightarrow S^{3}(-31) \oplus S(-27) \rightarrow S^{3}(-30) \oplus S^{3}(-25) \rightarrow S(-15)^{3} \rightarrow S,$$
so MacLane arrangement of conics is not nearly free. In fact, this arrangement is far from the class of nearly free arrangement since the defect is equal to
$$\nu(C) = \bigg\lceil \frac{3}{4} \cdot 15^{2} \bigg\rceil - 16 - 4 \cdot 32 = 25.$$

We made an extensive experimental search in order to find an appropriate (and natural) flat extension which would potentially allow us to construct a nearly free arrangement of $8$ conics from the MacLane arrangement of $8$ lines, but we failed. It would be extremely interesting to find a class of projective morphisms that could lead to nearly free arrangements of curves starting from nearly free (or free) line arrangements.

\section*{Acknowledgments}
We would like to warmly thank Giovanna Ilardi and Tomasz Szemberg for useful comments that allowed us to improve this note. The second author would like to thank Hal Schenck for insightful discussions about the content of the paper and beyond.
The first and second authors were partially supported by National Science Center (Poland) Sonata Grant Nr \textbf{2018/31/D/ST1/00177}.

Finally, we would like to thank an anonymous referee for careful reading our manuscript and for comments that allowed us to improve the paper.


Address: \\
Grzegorz Malara: Department of Mathematics, Pedagogical University of Cracow, 	Podchor\c a\.zych 2, PL-30-084 Krak\'ow, Poland. Email: grzegorz.malara$@$up.krakow.pl, grzegorzmalara@gmail.com \\\\
Piotr Pokora: Department of Mathematics, Pedagogical University of Cracow, 	Podchor\c a\.zych 2, PL-30-084 Krak\'ow, Poland. Email: piotr.pokora$@$up.krakow.pl, piotrpkr$@$gmail.com \\\\
Halszka Tutaj-Gasi\'nska: Faculty of Mathematics and Computer Science, Jagiellonian University, Stanis{\l}awa {\L}ojasiewicza 6, 30-348 Kraków,
halszka.tutajgasinska@gmail.com
\end{document}